\documentclass[12pt,a4paper]{article}

\usepackage[utf8]{inputenc}
\usepackage{indentfirst}
\usepackage{misccorr}
\usepackage{graphicx}
\usepackage{amsmath}
\usepackage{amssymb}
\usepackage{amsthm}
\usepackage{hyperref}
\usepackage{color}

\newif\ifhide
\hidefalse

\newlength{\wdth}

\newtheorem{theorem}{Theorem}

\newtheorem{Lemma}{Lemma}

\newtheorem{Remark}{Remark}

\title{{\normalsize\tt\hfill\jobname.tex}\\
On strong law of large numbers for pairwise independent random variables
}

\author{A.T. Akhmiarova\footnote{M.V. Lomonosov Moscow State University \& Institute for Information Transmission Problems (A.A.Kharkevich Institute); email: lili-r01@yandex.ru}, \; A.Yu. Veretennikov\footnote{Institute for Information Transmission Problems (A.A.Kharkevich Institute), email: ayv@iitp.ru}}

\date{}

\begin{document}
\maketitle

\begin{abstract}
\noindent
A new  version of a strong law of large numbers for pairwise independent random variables is proposed. The main goal is to relax the assumption of existence of the expectation for each summand. The assumption of pairwise independence is also relaxed.

\medskip


\medskip


\end{abstract}

\section{Introduction}
The Law of Large Numbers (LLN for what follows) is one of the basic theorems of probability and mathematical statistics. The interest to any new version of it is natural because the majority of  methods of justifying consistency in parameter estimation and of basic properties in hypothesis testing are based on the LLN in one form or another. Nowadays, more and more involved contemporary models emerge, which requires a regular revision of old methods. Moreover, further results on precision of the rate of convergence are often constructed on certain extensions of algorithms which are not optimal themselves but only consistent and, hence, are usually only LLN-based.  

The history of the LLN started with the work by Jakob Bernoulli in 1713, see \cite{Bernoulli}; then it was continued by de Moivre, Laplace, Chebyshev \cite{Chebyshev} et al. This history might have seemed to be completed with the discoveries by Khintchine \cite{Khintchin} and Kolmogorov \cite[theorem 6.3]{Kolmogorov} who found that both for the weak LLN (in \cite{Khintchin}) and for the strong LLN (in \cite{Kolmogorov}, SLLN in what follows) it suffices that the expectation of each of the summands are finite, in addition to 
the joint independence and identical distributions. Note that the first strong version of the LLN only appeared in the beginning of the 20th century \cite[chapter II, section 11]{Borel1909} and \cite[Lemma 2 \& formulae (28)-(29)]{Cantelli}. See \cite{Seneta92} about the historical role of these two papers published in the ``pre-Kolmogorov's'' epoch and thus not based on the modern probability axioms \cite{Kolmogorov} but at the same time very close to them. The work \cite{Mazurkiewicz} by Mazurkiewicz was the third one published at the same time as Cantelli's paper and also on the SLLN; it remained practically unknown for a very long time.

An  early version of the SLLN for only pairwise independent identically distributed summands  with a finite variance, was established in  \cite{Rajchman32}, just one year before the publication of the famous ``Grundbegriffe Der Wahrscheinlichkeitsrechnung'' by Kolmogorov in German in 1933 (the reference \cite{Kolmogorov} is given on the 2nd edition of its Russian translation). In 1937 the important paper by Marcinkiewicz and Zygmund \cite{MZ} was published, which concerned an SLLN version for independent and identically distributed summands possessing the moment of order $a\in (0,2)$.

\ifhide
{\color{blue}\fi
One more direction of the LLN and SLLN which must be also mentioned is the general ergodic theorems. Birkhoff \cite{Birk} established his ergodic theorem for dynamical system with an integral invariant. Further, Khintchin \cite{Khin-ergodic} gave his version explicitly for strictly stationary processes (in his paper convergence in probability was established) with a proper reference to the earlier important work by von Neumann. Then Kolmogorov published his simplified proof (which is different from his own SLLN proof) of what he called Birkhoff -- Khintchin ergodic theorem also for stationary sequences and processes. His result in \cite{Kolmogorov2} is about convergence almost surely and it does cover Kolmogorov's SLLN, thus extending it to dependent strictly stationary processes, although, not to all pairwise independent ones. And vice versa, notice that the pairwise independence case  does not fully cover the case of strictly stationary processes. 
%(Blank?!)
\ifhide}\fi

The  development of the WLLN and SLLN continues until nowadays. First of all, not touching all related works, let us 
point out the paper  \cite{Etemadi81} by Etemadi as one of the most important publications of the whole research area, in which a new approach for establishing the SLLN was offered. This approach is suitable not only for jointly independent random variables as in \cite[theorem 6.III]{Kolmogorov} but also for pairwise independent ones; moreover, it keeps the minimal assumption on the first moment of the summands and only uses Chebyshev's inequality instead of the more ``usual'' in this area martingale inequalities introduced by Kolmogorov.

This direction was further extended to the case of the non-equally distributed summands under certain additional conditions in the paper  \cite{Hungarians} by Cs\"orgo, Tandori, and Totik. Then followed the papers \cite{Rogge, Chandra, Rogge2} by Landers \& Rogge, and by Chandra, solely and with co-authors,  in which the conditions on the non-equally distributed summands were weakened to the uniform integrability (UI) and further to the Ces\`aro-UI; in \cite{Rogge2} the summands are assumed to be just non-correlated\footnote{Notice that under the assumption of finite variances a version of the SLLN may be found in some contemporary collections of exercises, see \cite[Exercise  4.24]{ZSCh} (of course, this exercise is marked here as a more difficult one).}. 

Different kinds of UI conditions have been proposed for weak LLN and for strong LLN. One of them, in the paper \cite{Chandra} suffices for a new version of the SLLN which extends Kolmogorov's result, moreover, allows for pairwise independent random variables; some other new versions of the SLLN were proved under strong UI conditions which were not enough to include the classical Kolmogorov's SLLN. Let us also highlight that the work \cite{Rajchman32} by Rajchman published originally in 1931 was briefly presented and extended in a relatively recent paper \cite{Chandra3}. Some other conditions allowed to establish new versions of the {\em weak} LLN, which will not be discussed in what follows. Also, some other SLLN extensions may be found in the paper \cite{bingham86}. Finally, in the recent publication \cite{Janisch} some new versions of sufficient conditions for the SLLN close to the ones in \cite{Chandra2} may be found. Let us add that more references may be found in the majority of the recent papers on the WLLN and SLLN; only a small part of the existing literature is cited in this paper.

{\bf Quick preview.}
The goal of the present paper is to offer some further  extension of the case of pairwise independent random variables, which includes Kolmogorov's SLLN \cite[theorem 6.5.3]{Kolmogorov}. The main contribution is that a certain infinite part of the summands may have infinite expectations, and that the random variables in this part  may be arbitrarily dependent on each other. Formally neither of these two features is allowed in any of the papers cited earlier. 
The  approach uses one of the theorems by Chandra from \cite{Chandra} and a new generalisation of the important auxiliary result from the paper \cite{Sawyer69} by Sawyer, which was briefly mentioned in the article \cite{Petrov} devoted to some other aspect of the LLN. References on other results and other directions in the area of the LLN may be found in historical parts of the cited literature. 
The present paper consists of four sections: this introduction, the setting, the SLLN for partially pairwise independent random variables (theorem 
\ref{thm1}), the proof of Theorem \ref{thm1}.

\section{The setting} 
Initially let us assume that the sequence $(X_n, n\ge 1)$ consists of  pairwise independent random variables  with finite expectations, for which without loss of generality the equalities 
\begin{equation}\label{eq2-1}
\mathsf E X_i = 0, \qquad i\ge 1,
\end{equation}
are satisfied. In what follows this requirement will be relaxed: finite  expectations will be assumed not for all summands.  The main result will be about another sequence $(Z_n)$ which will be constructed from $(X_n)$ and from one more complementary sequence of random variables $(Y_n)$.

Let us briefly describe what kind of a condition may be assumed in addition to (\ref{eq2-1}) for pairwise independent random variables, which would guarantee the SLLN for the sequence $(X_n)$. In this paper our choice will be condition (\ref{eq2-8}) in the end of this section.

Most recent extensions of the classical LLN for  pairwise independent random variables $X_n$ assume one or another form of the UI, either as 
\begin{equation}\label{eq2-2}
\sup_n \mathsf E|X_n| 1(|X_n|>N) \to 0, \quad N\to \infty, 
\end{equation}
or with Ces\`aro averages
\begin{equation}\label{eq2-3}
\sup_n \frac1{n} \sum_{k=1}^n \mathsf E|X_k| 1(|X_k|>N) \to 0, \quad N\to \infty,
\end{equation}
in addition to finite expectations of summands. 
One more condition which is also of the UI type, yet a little more restrictive, has the form (see \cite{BoseChandra94})
\begin{equation}\label{eq2-4}
\int_0^\infty \sup_n \mathsf P(|X_n|>x)dx < \infty.
\end{equation}
The interest and importance of the variant (\ref{eq2-4}) is because in the case of identically distributed summands it reduces to the classical Kolmogorov's condition, unlike other extensions of the SLLN for pairwise independent random variables. Conditions (\ref{eq2-2}) and (\ref{eq2-3}) in the absence of any other assumptions only suffice for a weak LLN \cite{Chandra, Rogge}.

It may be mentioned that if 
just a finite number of the first members of the sequence $(X_n)$ have infinite expectations, it, of course, does not affect the final asymptotic result. In turn, condition (\ref{eq2-4}) may be relaxed as follows: there exists $n_0$ such that 
\begin{equation}\label{eq2-7}
\int_0^\infty \underbrace{\sup_{n\ge n_0} \mathsf P(|X_n|>x)}_{=: \hat G^X(x)}dx < \infty,
\end{equation}
or, assuming a weaker,  Ces\`aro analogue of (\ref{eq2-7}),
\begin{equation}\label{eq2-8}
C(n_0):=  \int_0^\infty  \sup_{n\ge n_0} \frac1{n} \sum_{k=n_0}^n \mathsf P(|X_k|>x)dx < \infty. 
\end{equation}
The latter condition will be the main assumption on the sequence $(X_n)$ in the theorem of the next section, beside the pairwise independence. However, as it was already mentioned earlier, the theorem will tackle the SLLN not for the sequence $(X_n)$, but for another one denoted by $(Z_n)$, which is some mixture of $(X_n)$ with a second auxiliary sequence of random variables $(Y_n)$, where each $Y_n$ may have no finite expectation. More than that, the elements of the sequence $(Y_n)$ may be arbitrarily dependent. In the beginning of the next section it will be explained how the sequence $(Z_n)$ is constructed from $(X_n)$ and $(Y_n)$, and also the details of the assumptions about the sequence $(Y_n)$. 

\section{SLLN for partially pairwise independent random variables}
It is assumed that the sequence $(Z_n)$ consists of terms from two different  sequences $(X_i)$ and  $(Y_j)$, which possess different properties related to independence and existing moments; these properties are explained in what follows in this section.  The construction of $(Z_n)$ from $(X_i)$ and $(Y_j)$ means that for certain indices $n$ the random variable $Z_n$ is taken from the sequence $(X_i)$, while for the other, complementary indices $n$ this  random variable is chosen from the sequence $(Y_j)$. For each $n$ this choice is determined. It is assumed that where $X_i$ is chosen for the first time, it is $X_1$; where it is chosen for the second time, it is $X_2$, and so on. 
Similarly for the sequence $(Y_j)$. 
So, in this way, the sequence $(Z_n)$ may have a general outlook such as $(X_1,X_2,Y_1,X_3,X_4,X_5,Y_2,Y_3,X_6,\ldots)$.  

A more formal way to define the sequence $(Z_n)$ is as follows. Let us fix a {\bf deterministic} sequence of zeros and ones $(\alpha(n), \, n\ge 1)$. Thus, for each $n$, 
$\alpha(n)=0$, or $\alpha(n)=1$, and this choice
is deterministic, not random. Let
$$
\kappa(n) : = \sum_{k=1}^n \alpha(k), \quad \psi(n) = n- \kappa(n), \quad n\ge 1.
$$
Now, given two sequences  $(X_n)$ and $(Y_n)$, the construction of the sequence $(Z_n)$ reads, 
$$
Z_n =\left\{
\begin{array}{cc}
X_{\psi(n)}, & \alpha(n)=0, \\
Y_{\kappa(n)}, & \alpha(n)=1.
\end{array}
\right.
$$
In this way, for each $n$ the random variable $Z_n$ equals either some $X_i$, or some $Y_j$, depending on the value of $\alpha(n)$.

Recall that the random variables from the sequence $(X_n)$ are pairwise independent, satisfy a Ces\`aro-UI type condition (\ref{eq2-8}) with some $n_0\ge 1$, and they are also integrable and   
possess zero expectations starting from the index~$n_0$, 
\begin{equation}\label{E0}
\mathsf EX_n=0, \quad n\ge n_0. 
\end{equation}
Note that condition (\ref{eq2-8}) automatically implies that for $n\ge n_0$ expectations of all $X_n$ are finite.

The  sequence of random variables  $(Y_n)$ admits the property 
\begin{equation}\label{3-A1}
\mathsf E|Y_n|^a <\infty, \quad \text{with some  $a\in (0,1)$ for all $n\ge 1$},
\end{equation}
and, moreover, there exists $\varepsilon \ge 0$ such that for any $n\ge 1$  
\begin{equation}\label{3-A2b}
%\sup_{n\ge 1} 
\frac1{n}\sum_{k=1}^n \mathsf P(|Y_k|^a >x) \le n^\varepsilon\tilde G^Y(x), \;\text{where} \;\int_0^\infty \tilde G^Y(x)dx\le C <\infty.
\end{equation}

Finally, it is required that the places where $\alpha(n)=1$ should be  ``rare'' in the following sense:  
\begin{equation}\label{3-A3}
\limsup_{n\to\infty}  \frac{\kappa(n)}{n^{\frac{a}{1+a\varepsilon}}} <\infty,
\end{equation}
where $a, \varepsilon$ are the same constants as in conditions (\ref{3-A1})-(\ref{3-A2b}). 
The main result of this paper is the following version of the SLLN.

\medskip

\begin{theorem}\label{thm1}
Under the assumptions 
(\ref{eq2-8})   -- (\ref{3-A3}) 
the following SLLN holds true: 
$$
\frac1{n} \sum_{k=1}^n Z_k \stackrel{\text{a.s.}}\to 0.
$$
\end{theorem}

\begin{Remark}
Notice that both the UI and Ces\`aro-UI {\em for all} $Z_n$, generally speaking, may not be satisfied. In comparison with many of the earlier works the condition of pairwise independence is relaxed. Namely, no independence is required for the ``bad'' rare  sequence $(Y_n)$. 
Theorem \ref{thm1} generalises earlier results from \cite{Kolmogorov}, \cite{Etemadi81},  \cite{BoseChandra94}, including in the case $\varepsilon = 0$.   
In the latter case of $\varepsilon = 0$  condition (\ref{3-A2b}) becomes a version of the Ces\`aro-UI for the r.v. $|Y_k|^a$, which also generalises earlier used conditions. 
\end{Remark}

\section{Proof of Theorem \ref{thm1}}
Two lemmata will be required for the proof. The first one is practically a result established in \cite{BoseChandra94} and \cite{Korchevsky}, and  is an analogue of Etemadi's  theorem from \cite{Etemadi81}. This is why its short proof is presented with the reference to \cite{BoseChandra94, Korchevsky} just for the completeness and for the convenience of the reader. 
The second lemma is an analogue of the corollary to the lemma 3 in the paper \cite{Sawyer69} for non-equally distributed random variables. The statement of lemma 2 and its proof are both new, although, both are based on the results from \cite{Sawyer69} which are pointed out.

\begin{Lemma}[for $(X_n)$]\label{lem1}
Let conditions (\ref{eq2-8}) -- (\ref{E0}) be satisfied. Then 
\begin{equation}\label{eqlem2-1}
\frac1{n} \sum_{k=1}^{n-\kappa(n)} X_k \stackrel{\text{a.s.}}\to 0, \quad n\to \infty.
\end{equation}
\end{Lemma}

\begin{proof}
The moment properties of the first $n_0-1$ summands do not affect the limit in (\ref{eqlem2-1}), since
\[
\frac1{n}\sum_{k=1}^{n_0-1} X_k \stackrel{\text{a.s.}} \to 0, \quad n\to\infty,
\]
without any assumption. 
Hence, it suffices to check that 
\[
\frac1{n} \sum_{k=n_0}^{n-\kappa(n)} X_k \stackrel{\text{a.s.}}\to 0, \quad n\to \infty.
\]
This property and, therefore, the statement of the lemma follow from theorem 2 in \cite{Korchevsky}; see also \cite{BoseChandra94}. 
\end{proof}

\begin{Lemma}[for $(Y_n)$]\label{lem3}
Let there exist $a\in (0,1)$ and $\varepsilon\ge 0$ such that 
condition (\ref{3-A2b}) holds true for any $n\ge 1$ and condition  (\ref{3-A3}) is satisfied.
Then
\begin{equation}\label{eqYn0}
\frac1{n} \sum_{k=1}^{\kappa(n)} Y_k \stackrel{\text{a.s.}}\to 0, \quad n\to \infty.
\end{equation}
\end{Lemma}

\begin{proof}
Let 
$V_n:= |Y_n|^a$ and $\mathsf P(|Y_n|^a >x) =:\bar G_n^Y(x)$; notice that due to the condition (\ref{3-A1}) all random variables $V_n$ have finite, although, possibly not uniformly bounded expectations: 
$$
\mathsf EV_n \le n^{1+\varepsilon}\int_0^\infty \tilde G^Y(x)dx < \infty.
$$
Let $\tilde{V}_n:=\min (V_n, n^{1 +\varepsilon/p}) $, where $p>1$. We have, following the paper \cite[the proof of lemma 3]{Sawyer69} and some hints from earlier works, in particular,  \cite{MZ}, see also \cite{Korchevsky},  with  
$m= \lceil n^{1+\varepsilon a}\rceil$,  $\varepsilon' = \varepsilon a$, and using that 
$$
\frac{\lceil n^{1+\varepsilon/p}\rceil ^p}{n^{p+\varepsilon}} \le
\frac{(n^{1+\varepsilon/p} +1 )^p}{n^{p+\varepsilon}} 
\le \frac{(n +1 )^{p+\varepsilon}}{n^{p+\varepsilon}} \le 2^{p+\varepsilon}, \quad n\ge 1,
$$
and so, 
$$
\frac{2^{p+\varepsilon}}{n^{p+\varepsilon}} \le
\frac{2^{p+\varepsilon}}{\lceil n^{1+\varepsilon/p}\rceil ^p} \le 
\frac{2^{p+\varepsilon}}p 
\sum_{j=m}^\infty \frac1{j^{p+1}} 
= \frac{2^{p+\varepsilon}}p
\sum_{j=\lceil n^{1+\varepsilon/p}\rceil}^\infty \frac1{j^{p+1}}, 
$$
we estimate, 
\begin{align*}
%1
& \mathsf E\sum_{n=3}^{\infty}\frac{\tilde{V}^p_n}{n^{ p+\varepsilon}}
=\sum_{n=3}^{\infty}\frac1{n^{p+\varepsilon}}\mathsf E(V^{p}_n\wedge n^{p+\varepsilon})
 \\\\
%2 
& = \sum_{n=3}^{\infty}\frac1{n^{p+\varepsilon}} 
\mathsf E \int_0^{\infty} p t^{p-1} 1(t<V_n\wedge n^{1+\varepsilon/p})dt
\\\\
%3 
& = \sum_{n=3}^{\infty}\frac1{n^{p+\varepsilon}} 
\mathsf E \int_0^{n^{1+\varepsilon/p}} p t^{p-1} 1(t<V_n)dt
 \\\\
%4
& = \sum_{n=3}^\infty \frac1{n^{p+\varepsilon}} 
\int_0^{n^{1+\varepsilon/p}} p t^{p-1} \bar G^Y_n(t)dt 
\le \sum_{n=3}^\infty \frac1{n^{p+\varepsilon}} 
\int_0^{\lceil n^{1+\varepsilon/p}\rceil} p t^{p-1} \bar G^Y_n(t)dt 
 \\\\
%5
& \le c_{p,\varepsilon}\sum_{n\ge 3} 
\underbrace{\sum_{j=m}^\infty \frac1{j^{p+1}}}_{= c m^{-p} = cn^{-(p+\varepsilon)}} 
\sum_{\ell =1}^{\lceil n^{1+\varepsilon/p} \rceil =m} \int_{\ell-1}^\ell p t^{p-1} \bar G^Y_n(t)dt
 \\\\
%6
& \le c_{p,\varepsilon}\sum_{n\ge 3} 
\underbrace{\sum_{j=m}^\infty \frac1{j^{p+1}}}_{= c m^{-p} = cn^{-(p+\varepsilon)}} 
\sum_{\ell =1}^{j} \int_{\ell-1}^\ell p t^{p-1} \bar G^Y_n(t)dt
 \\\\
%7
& \le 
c_{p,\varepsilon} \sum_{j\ge 3} \frac1{j^{p}} \sum_{\ell =1}^{j}  \int_{\ell-1}^\ell p t^{p-1} \underbrace{\frac1{j} \sum_{n=3}^{\lceil j^{p/(p+\varepsilon)} \rceil}\bar G_n^Y(t)}_{\le \tilde G^Y(t)}dt 
 \\\\
%8 
& \le c_{p,\varepsilon} \sum_{j\ge 3} \frac1{j^{p}} \sum_{\ell =1}^{j}  \int_{\ell-1}^\ell p t^{p-1} \tilde G^Y(t)dt.
\end{align*}

It was used that 
$j \ge  n^{1+\varepsilon/p}\,\Longleftrightarrow \, 
n \le j^{1/(1+\varepsilon/p)} \equiv  j^{p/(p+\varepsilon)}$. 
Further, 
\begin{align*}
&
c_{p,\varepsilon} \sum_{j\ge 3} \frac1{j^{p}} \sum_{\ell =1}^{j}  \int_{\ell-1}^\ell p t^{p-1} \tilde G^Y(t)dt 
 \\\\
& = c_{p,\varepsilon} \sum_{\ell = 1}^\infty \int_{\ell-1}^\ell p t^{p-1} \tilde G^Y(t)dt 
\underbrace{\sum_{j=3\vee \ell} ^\infty \frac1{j^{p}}}_{\displaystyle\le \frac{c_p}{(\ell)^{p-1}}}
\le  c_{p,\varepsilon} \int_0^\infty \tilde G^Y(t)dt <\infty.
\end{align*}
All constants $c_{p,\varepsilon}$ in the latter two calculi are finite and may be equal or not equal to one another. Also, notice that the first two terms of the series 
$
\mathsf E\sum_{n=1}^{\infty}n^{-p- \varepsilon}\tilde{V}^p_n
$
are finite. Therefore, the series 
$$
\sum_{n=1}^{\infty}\frac{\tilde{V}^p_n}{n^{p +\varepsilon}}<\infty,  
$$
converges almost surely along with 
$$
\sum_{n=1}^{\infty}\frac{V^p_n}{n^{p +\varepsilon}}
=\sum_{n=1}^{\infty}\frac{|Y_n|^{ap}}{n^{p +\varepsilon}}<\infty \quad (\text{a.s.}).  
$$

Indeed, due to the choice  
$\tilde{V}_n:=\min (V_n, n^{1 +\varepsilon/p}) $ we have, 
$$
\sum_{n=1}^{\infty}\frac{\min^p (V_n, n^{1 +\varepsilon/p})}{n^{p +\varepsilon}}
\equiv \sum_{n=1}^{\infty}\frac{\min (V_n^p, n^{p +\varepsilon})}{n^{p +\varepsilon}}
\equiv \sum_{n=1}^{\infty}
\min\left(\frac{V_n^p}{n^{p +\varepsilon}}, 1\right)
<\infty. 
$$
This implies that the number of terms which are equal to 1 is finite, otherwise the series would diverge. So, only a finite number of summands will be replaced by $1$, 
which will not affect the convergence of the series $\displaystyle 
\sum_{n=1}^{\infty}\frac{V^p_n}{n^{p +\varepsilon}}$.

\medskip

Now choosing $p=1/a$ we obtain
\[
\sum_{n=1}^{\infty}\frac{|Y_n|}{n^{a^{-1}  +\varepsilon}}<\infty \quad (\text{a.s.}).
\]
By virtue of Kronecker's lemma (see, for example, \cite[Lemma IV.3.2]{Shiryaev2}), it follows that 
\[
\frac1{n^{ \varepsilon+1/a}} \sum_{k=1}^{n} Y_k \stackrel{\text{a.s.}}\to 0, \quad n\to \infty,
\]
which implies (\ref{eqYn0}) due to the condition (\ref{3-A3}). 
The statement of the  lemma follows.
\end{proof}

\medskip
\noindent
{\em Proof of Theorem \ref{thm1}.}
The goal is establishing the convergence 
\begin{equation*}\label{eq4-1}
\mathsf P\left(\frac1{n} \sum_{k=1}^n Z_k 
\to 0, \quad n\to \infty\right) = 1. 
\end{equation*}
The required equality follows straightforwardly from (\ref{eqlem2-1}) and  (\ref{eqYn0}). Theorem \ref{thm1} is proved \hfill 
$\square$

\section*{Acknowledgements}
\noindent
For both authors this study was funded by the Theoretical Physics and Mathematics Advancement Foundation ``BASIS''. The authors are grateful to the anonymous Referee for useful remarks. 
\ifhide
{\color{blue}\fi The authors are also grateful to Jan Palczewskii and Tadeusz
Banek for help with access to less-available sources in Polish. The authors are thankful to the participants of the online seminar ``Probabillity and mathematical statistics, the three cities seminar'' for the discussion about links of the SLLN with Birkhoff -- Khintichin's ergodic theorem, and particular thanks to Michael Blank, Kostya Borovkov, Valery Oseledets, Sergey Pirogov,  Elena Zhizhina. A special debt of gratitude to Dancho Stoyanov for the comments to the two preprint versions (https://doi.org/10.48550/arXiv.2211.04435).
\ifhide
}\fi

\end{document}